\documentclass[a4paper,12pt]{article}

\usepackage[english]{babel}
\usepackage{amsmath,amsfonts,amssymb,amsthm}

\usepackage[top=2cm, bottom=2cm, left=2.5cm, right=2.5cm]{geometry}

\usepackage{indentfirst}
\usepackage[T1]{fontenc} 

\newtheorem{theorem}{Theorem}[section] 
\newtheorem{lemma}[theorem]{Lemma}     
\newtheorem{definition}[theorem]{Definition}

\numberwithin{equation}{section}

\newcommand{\R}{\mathbb R}
\newcommand{\C}{\mathbb C}
\newcommand{\Cd}{\overline{\mathbb C}}
\renewcommand{\O}{\mathcal O}
\newcommand{\T}{\mathcal T}
\newcommand{\Ta}{\mathrm T}

\newcommand{\diam}{\mathrm{diam}}

\begin{document}

\begin{center}
\Large
\textbf{No entire function with real multipliers in class $\mathcal{S}$}\\

\bigskip
\large
Agnieszka Bade\'nska

\smallskip
\footnotesize
Faculty of Mathematics and Information Science\\
Warsaw University of Technology\\
ul. Koszykowa 75\\
00-662 Warszawa\\
Poland\\
badenska@mini.pw.edu.pl 
\end{center}

\normalsize
\begin{abstract}
We prove that there is no entire transcendental function in class~$\mathcal{S}$ with real multipliers of all repelling periodic orbits.
\end{abstract}

\section{Introduction}

Already Fatou in~{\cite[section 46]{fatou}} pointed out that if the~Julia set of a~rational function is a~smooth curve, then all periodic points in the~Julia set have real multipliers. Recently Eremenko and van~Strien proved the~converse statement {\cite[Theorem~1]{evs}}, i.e. if all repelling periodic orbits of a~rational function $f\colon\Cd\to\Cd$ (of degree at least~2) have real multipliers, then either the~Julia set is contained in a~circle or $f$~is a~Latt\'es map. Moreover, they gave a~detailed description of these rational maps whose Julia set is contained in a~circle.

Their result (as well as our previous considerations in~{\cite[Lemma~3.7]{rigidity}}) was a~motivation for us to study transcendental functions with real multipliers of all repelling periodic orbits. We prove the~following claim:
\begin{theorem}\label{main}
There exists no entire transcendental function in class~$\mathcal{S}$ with real multipliers of all repelling periodic orbits.
\end{theorem}

In order to prove it, we will construct on~$\C$ an~invariant line field, univalent in a~neighbourhood of the~Julia set. This will contradict Theorem~6.1 from~\cite{rvs} saying that such a~line field cannot exist in the~transcendental case. The~construction of a~line field uses a~global linearization function associated with a~repelling periodic point that posses only regular preimages. We will show that this linearization function is automorphic with respect to a~discrete group of affine maps. The~proof of this part strongly relies on methods introduced by Eremenko and van~Strien in~{\cite[Sections 1.1-1.2]{evs}}. 

Note that to obtain the~contradiction it is enough to assume that $f$~has real multipliers of repelling cycles in some relatively open set in the~Julia set. Let us also remark that the~proof remains valid if $f$ is an~entire transcendental function from class~$\mathcal{B}$ (i.e. its finite singular values form a~bounded set) for which there exists in the~Julia set a~relatively open set disjoint from orbits of singular values.

\section{Preliminaries}

Throughout the~entire article $f\colon\C\to\C$ denotes an~entire transcendental function from class~$\mathcal{S}$, i.e. with finitely many singular values. By singular value we mean critical value or finite asymptotic value. Note that in this case, similarly as for rational functions, only finitely many periodic points may belong to the~forward orbits of singular values. It~means that from among an~infinite number of repelling cycles we can always choose such that have only regular preimages.

Recall that the~Fatou set $F(f)$ consists of points in whose neighbourhood the~iterates~$\{f^n\}_{n\geq1}$ form a~normal family while the~Julia set is its complement $J(f)=\C\setminus{F(f)}$. The~Julia set may be equivalently defined as the~closure of all repelling periodic orbits of~$f$. Moreover, if $z\in{J(f)}$ is not an~exceptional value (i.e. its backward orbit is infinite), then
\begin{displaymath}
	J(f)=\overline{\bigcup_{n\geq1}f^{-n}(z)}.
\end{displaymath}
For these and more properties of Fatou and Julia sets see~\cite{berg}. We will also use the~fact that the~Hausdorff dimension of the~Julia set of an~entire transcendental function with finitely many singularities is greater than one (see~\cite{stallard}).

Our proof uses the~notion of invariant curves and unstable manifolds introduced in~{\cite[Section~1.1]{evs}}.
\begin{definition}
	A~simple curve $\gamma\colon(0,1)\to\C$ passing through a~repelling periodic point~$p$ of period~$n$ is called an~unstable manifold for~$p$ if there exists a~subarc $\gamma_*\subset\gamma$ containing~$p$ such that $f^n$~maps $\gamma_*$ diffeomorphically onto~$\gamma$.
	
	We say that $\gamma$~is an~invariant curve for $q\in{f^{-m}(p)}$ if $\gamma$~is contained in an~unstable manifold for~$p$, $q\in\gamma$ and $f^m(\gamma\cap{V})\subset\gamma$ for some neighbourhood~$V\ni{q}$.
\end{definition}
Note that this is not what we usually mean by an~unstable manifold for a~hyperbolic periodic point of a~smooth dynamical system. We decided, however, to apply the~notion from~\cite{evs} for the~convenience of the~reader, since our proof closely follows their.

As we mentioned in the~introduction our first aim will be to show that if $f$~has only real multipliers of repelling periodic orbits, then a~linearization function~$\Psi$ associated with a~repelling point (with regular preimages) is automorphic with respect to a~discrete group of isometries of the plane~$\C$ (see {\cite[Definition~2.4]{mayer}}).
\begin{definition}\label{aut}
A~holomorphic map $\Psi\colon\C\to\Cd$ is called automorphic with respect to a~discrete group of isometries $\Gamma\subset\mathrm{Isom}(\C)$ if:
\begin{itemize}
	\item[(1)] $\Psi\circ\gamma=\Psi$ for every $\gamma\in\Gamma$ and
	\item[(2)] $\Gamma$ acts transitively on fibers, i.e. whenever $\Psi(z_1)=\Psi(z_2)$, there is $\gamma\in\Gamma$ with $\gamma(z_1)=z_2$.
\end{itemize}
\end{definition}
It was shown in {\cite[Lemma~2.5]{mayer}} that the transitivity condition \emph{(2)} in the above definition can be replaced by the following weaker assertion:
\begin{itemize}
	\item[\emph{(2')}] \emph{$\Gamma$ acts transitively on one regular fibre, i.e. there is $w\in\Psi(\C)$ which is not a~critical value of~$\Psi$ such that for every pair $z_1,z_2\in\Psi^{-1}(w)$ we can find some $\gamma\in\Gamma$ with $\gamma(z_1)=z_2$.}
\end{itemize}

We will use the~linearization function to construct an~invariant line field on~$\C$ univalent in an~open set intersecting the~Julia set (cf.~{\cite[Definition~4.1]{rvs}}). Such line field cannot exist for transcendental functions by the~result of Rempe and van~Strien {\cite[Theorem~6.1]{rvs}}. Recall that an~invariant line field on~$\C$ is a~measurable choice of directions for any $z\in\C$ invariant under~$f$. Since any~direction (straight line through the~origin) is determined by two angles which differ in~$\pi$, it is convenient to associate an~invariant line field with a~measurable function $u\colon\C\to\mathrm{S}^1$ satisfying the~invariance condition
\begin{displaymath} u(f(z))=\left(\frac{f'(z)}{|f'(z)|}\right)^2u(z) \end{displaymath}
for almost all $z\in\C$.

\section{Proof of Theorem~\ref{main}}

Suppose $f\colon\C\to\C$ is an~entire transcendental function from class~$\mathcal{S}$ such that all repelling cycles have real multipliers. Choose a~repelling periodic point~$p\in{J(f)}$ which does not belong to the~forward orbit of singular values. Replacing if necessary $f$ by its iterate we may assume that $p$ is a~fixed point. 

Denote by~$\Psi$ a~local linearization of~$f$ at~$p$, i.e. a~univalent function defined on a~small disc~$\O_0$ centered at the~origin such that for $z$ close to~$0$,
\begin{equation}\label{lineariz} 
	\Psi(\lambda z)=f\circ\Psi(z), 
\end{equation}
where $\lambda=f'(p)$ is real, $\lambda>1$. Note that $\Psi$ is defined up to $\Psi'(0)\in\C\setminus\{0\}$ which can be chosen arbitrarily. We can extend $\Psi$ to a~global linearization $\Psi\colon\C\to\C$ satisfying the same relation~(\ref{lineariz}) by the~formula
\begin{displaymath} 
	\Psi(z)=f^n(\Psi(\lambda^{-n}z)), \quad n\geq1. 
\end{displaymath}
Note that, since $\Psi$ is univalent on~$\O_0$, it immediately implies the~following fact.
\begin{lemma}\label{regpreimages}
	If $p$ is any~repelling periodic point which is not an~iterate of a~singular value and $\Psi$ is a~global linearization at~$p$, then for any $z\in\Psi^{-1}(p)$ we have $\Psi'(z)\neq0$.
\end{lemma}

Choose any $Q\in\Psi^{-1}(p)\setminus\{0\}$, then $\Psi'(Q)\neq0$, so we can find a~topological disc $\O_1\ni{Q}$ on which $\Psi$ is univalent and such that $\Psi(\O_0)=\Psi(\O_1)$ (decreasing $\O_0$ if necessary). Hence, there exists a~biholomorphic map $\T\colon\O_0\to\O_1$ such that
\begin{displaymath} \T(0)=Q \quad \textrm{and} \quad \Psi\circ\T=\Psi \quad \textrm{on} \; \O_0. \end{displaymath}
Additionally, since $\O_0$ was chosen to be a~round disc, $\lambda^{-1}\O_0\subset\O_0$.

We are going to prove that $\T$ is actually of the~form $\pm{z}+Q$. To do this we will show that $\T'$ is real on an~infinite family of straight lines passing through the~origin. This part of the~proof closely follows \cite[sections~1.1--1.2]{rvs} however we needed to adapt it to the~case of entire transcendental function.  For the~convenience of the~reader we present it with all details and keep similar notation.

\begin{lemma}\label{seq}
For each $Q\in\Psi^{-1}(p)\setminus\{0\}$, there is a~sequence $z_n\to0$ such that $\lambda^nz_n\to{Q}$ and $\Psi(z_n)$ is a~repelling point of period~$n$. There exists a~neighbourhood $V_n\subset\O_0$ of~$z_n$ such that $f^n\colon\Psi(V_n)\to\Psi(\O_0)$ is biholomorphic.
\end{lemma}

\begin{proof}
Let $V_n=\lambda^{-n}\O_1$ and take $n$ so large, say $n\geq{N}$, that $\overline{V_n}\subset\O_0$. Note that
\begin{displaymath} 
	f^n|_{\Psi(V_n)}=\Psi|_{\O_1}\circ\lambda^n\circ(\Psi|_{V_n})^{-1} 
\end{displaymath}
is univalent. Moreover,
\begin{displaymath} 
	f^n(\Psi(V_n))=f^n(\Psi(\lambda^{-n}\O_1))=\Psi(\O_1)=\Psi(\O_0),
\end{displaymath}
hence there exists a~point $z_n\in{V_n}$ such that $\Psi(z_n)$ is a~repelling fixed point for~$f^n$. We obtain this way a~whole sequence $z_n\to0$, for $n\geq{N}$, as stated in the~lemma.
\end{proof}

\begin{lemma}\label{seqL}
Let $p$, $Q$, $\Psi$ and $\T\colon\O_0\to\O_1$ be as above. Denote by $L$ the~straight line through $0$ and~$Q$ and let $\gamma=\Psi(L\cap\O_0)$. Then:
	\begin{enumerate}
		\item The~curve $\gamma$ is an~unstable manifold for~$p$;
		\item For $n$ large enough there exists a~sequence $z_n\in{L}$ such that $z_n\to0$, $\lambda^n{z_n}\to{Q}$, $\Psi(z_n)$ is a~repelling point of period~$n$ and $\gamma$~is an~unstable manifold for~$\Psi(z_n)$;
		\item For $n$ large enough $\gamma$ is an~invariant curve for $\Psi(\lambda^{-n}Q)\in{f^{-n}(p)}$;
		\item $\Psi(L\cap\O_0)=\Psi(L\cap\O_1)$ and thus $\T\colon{L\cap\O_0}\to{L\cap\O_1}$ is a~diffeomorphism;
		\item Either $\T'$ is constant or $\{z\in\O_0:\T'(z)\in\R\}$ is a~finite union of real analytic curves, one of which is~$L\cap\O_0$.
	\end{enumerate}
\end{lemma}

\begin{proof}
Let $z_n$ be the~sequence defined in the~proof of Lemma~\ref{seq} and let $x_n=\Psi(z_n)$ be the~corresponding periodic point. We are going to show that
\begin{equation}\label{tprimznreal}
	\T'(z_n)\in\R.
\end{equation}
Since $\Psi\circ\lambda^n=f^n\circ\Psi$, hence $\Psi'(\lambda^nz_n)\lambda^n=(f^n)'(x_n)\Psi'(z_n)$. Recall that multipliers of all repelling cycles are real, thus $(f^n)'(x_n)\in\R$, hence $\Psi'(z_n)/\Psi'(\lambda^nz_n)\in\R$. 

Note also that $\Psi(\lambda^nz_n)=f^n(\Psi(z_n))=f^n(x_n)=x_n=\Psi(z_n)$. But for $n$ large enough we have: $z_n\in\O_0$, $\lambda^nz_n\in\O_1$, and since $\T\colon\O_0\to\O_1$~is biholomorphic with $\Psi\circ\T=\Psi$, we conclude that
\begin{equation}\label{Tzn}
	\T(z_n)=\lambda^nz_n. 
\end{equation}
Now, $\Psi'(z_n)=\Psi'(\T(z_n))\T'(z_n)=\Psi'(\lambda^nz_n)\T'(z_n)$ which implies~(\ref{tprimznreal}).

If $\T'$ is constant then $\T$~is an affine map, $\T(z)=az+Q$, where $a=\T'(0)\in\R$ and the~identity $\Psi\circ\T=\Psi$ extends to the~whole plane~$\C$. But this is possible only if $|a|=1$, so we conclude that $\T(z)=\pm{z}+Q$. Then obviously $\T(L)=L$, where $L$~is the~line through $0$ and~$Q$, in particular $\T(L\cap\O_0)=L\cap\O_1$. Moreover, (\ref{Tzn})~implies that $z_n=\frac{Q}{\lambda^n\pm1}\in{L}$. 

Suppose now that $\T'$ is not constant, then the~set $X=\{z\in\O_0:\T'(z)\in\R\}$ is a~finite union of real analytic curves. We will show that one of these curves is~$L\cap\O_0$. For simplicity we may assume without loss of generality that $L$ is the~real line, thus $Q\in\R$.

Let $\beta$ be a~curve in~$X$ containing infinitely many points~$z_n$. Since $\lambda^nz_n\to{Q}\in\R$, thus $\arg(z_n)\to0$ and we conclude that $\beta$~is tangent to $L=\R$ at~$0$. Note that if $\T|_{\R}$ is real, then we are done. Suppose that $\T|_{\R}$ is not real, then
\begin{displaymath} 
	\T(z)=Q+a_1z+a_2z^2+\ldots+a_mz^m+a_{m+1}z^{m+1}+O(z^{m+2}), \quad z\to0, 
\end{displaymath}
where $m$ is chosen so that $a_1,\ldots,a_m\in\R$ and $a_{m+1}\notin\R$. Since $a_1=\T'(0)\in\R\setminus\{0\}$, we have that $m\geq1$.

We are going to show that $\beta\subset\R$. Suppose on the contrary that this is not the~case and recall that $\beta$~is tangent to $L=\R$ at~$0$ so it is of the~form
\begin{displaymath}
	\beta(x)=x+ibx^{K}+o(x^K) \quad \textrm{for some} \quad b\neq0 \quad \textrm{and} \quad K>1.
\end{displaymath}
Let $k\geq2$ be the~smallest subscript for which $a_{k}\neq0$. Since $a_{m+1}\notin\R$, therefore $k\leq{m+1}$. 
If $k=m+1$, then the~condition $T'|_{\beta}\in\R$ implies $\Im\T'(\beta(x))=(m+1)\Im{a_{m+1}}x^m+\ldots\equiv0$, what is impossible since $\Im{a_{m+1}}\neq0$. It follows that~$k\leq{m}$.

Using that $(x+ibx^K+o(x^K))^{k-1}=x^{k-1}+(k-1)x^{k-2}ibx^K+o(x^{k-2+K})$, the~condition $T'|_{\beta}\in\R$ gives
\begin{displaymath}
	\Im\T'(\beta(x))=ka_k(k-1)bx^{k-2+K}+\ldots+(m+1)\Im{a_{m+1}}x^m+\ldots\equiv0.
\end{displaymath}
Note that, since $a_k\neq0$ and $\Im{a_{m+1}}\neq0$, this is possible only if $k-2+K=m$. Therefore,
\begin{equation}\label{mK}
	m\geq{K}
\end{equation}
because $k\geq2$.

Since $z_n\in\beta$, it is of the~form $z_n=t_n+ibt_n^K+o(t_n^K)$, where $t_n\to0$ and so $\arg{z_n}\sim{t_n^{K-1}}$. On the~other hand, we have that $\Re\T(z_n)\to{Q}\neq0$ and $\Im\T(z_n)=O(t_n^K)$ (in view of~(\ref{mK}) and the~form of~$\T$), therefore $\arg\T(z_n)=O(t_n^{K})$. But this clearly contradicts~(\ref{Tzn}) what proves the~second part of point~\textit{5}. And since any curve from~$X$ containing infinitely many points~$z_n$ is actually the~line~$L$, therefore for all~$n$ bigger than some~$N$, we have $z_n\in{L}$.

Now, $\T\colon\O_0\to\O_1$ is biholomorphic with $\T(0)=Q$ and $\T'|_{L\cap\O_0}$ is real, thus for any~$z\in{L}\cap\O_0$, we have $\T(z)\in{L}$, giving~\textit{4}.

If we take $\gamma_*=\Psi(\lambda^{-1}(L\cap\O_0))$, then it is clear that the~curve $\gamma=\Psi(L\cap\O_0)$ is an~unstable manifold for~$p$. In order to prove~\textit{2.}, recall that $f^n\colon\Psi(V_n)\to\Psi(\O_0)$ is biholomorphic and define $\gamma_*'=\Psi(L\cap{V_n})$, where $V_n=\lambda^{-n}\O_1$ comes from Lemma~\ref{seq}. Then $f^n(\gamma_*')=\Psi(\lambda^n(L\cap{V_n}))=\Psi(L\cap\O_1)=\gamma$. The~same argument proves also~\textit{3.}, since $\lambda^{-n}Q\in{L}\cap{V_n}$.
\end{proof}

Recall that $Q$ was an~arbitrary $\Psi$-preimage of~$p$ different from~$0$ and $\T\colon\O_0\to\O_1$ was a~biholomorphic map such that $\Psi\circ\T=\Psi$ on~$\O_0$. So far we showed that $\T'$~is real on the~line~$L$, in what follows we are going to prove that $T'$~is actually constant.

\begin{lemma}\label{Taffine}
	For any $Q\in\Psi^{-1}(p)\setminus\{0\}$, we have $\T(z)=\pm{z}+Q$ and $\Psi\circ\T=\Psi$ on~$\C$.
\end{lemma}

\begin{proof}
Notice that $\Psi^{-1}(J(f)))$ cannot be contained in a~finite family of lines since the~Hausdorff dimension of~$J(f)$ is greater than one (see~\cite{stallard}). In particular $\Psi^{-1}(J(f))\nsubseteq{L}$. Recall that $p$~is not an~asymptotic value so it is not exceptional, hence its backward orbit is dense in~$J(f)$. There exists thus a~point $Q^1\in\C\setminus{L}$ such that $f^m(\Psi(Q^1))=p$ for some $m\geq1$. Define $Q'=\lambda^m{Q^1}$, then
\begin{displaymath}
	\Psi(Q')=\Psi(\lambda^m{Q^1})=f^m(\Psi(Q^1))=p
\end{displaymath}
and denote by $L'$ the~line through $0$ and~$Q'$. Note that we can construct infinitely many such lines. We will show that $\T'$ is real on~$L'$.

By~Lemma~\ref{regpreimages} we have $\Psi'(Q')\neq0$ so we can choose $\O_0'$ and $\O_1'$, neighbourhoods of $0$ and $Q'$ respectively, on which $\Psi$~is univalent and such that $\Psi(\O_0')=\Psi(\O_1')$. As before we want $\O_0'$ to be a~round disc. Applying Lemma~\ref{seq} and \ref{seqL} to the~point~$Q'$ we obtain a~sequence $z_k'\in{L'}$ such that $z_k'\to0$, $\lambda^kz_k'\to{Q'}$ and $x_k'=\Psi(z_k')$ is a~repelling point for~$f$ of period~$k$. Fix a~point $z=z_k'$ so that $\lambda^kz\in\O_1'$ and $x=\Psi(z)$, the~corresponding periodic point, does not belong to the~forward orbit of singular values. Note that there are infinitely many points~$z$ satisfying these conditions and we will show that for any such point, $\T'(z)\in\R$.

By Lemma~\ref{seqL}, the~curve $\gamma'=\Psi(L'\cap\O_0')=\Psi(L'\cap\O_1')$ is an~unstable manifold for~$p$ and for~$x$. Note that it is connected, smooth and has no intersections ($\Psi|_{\O_0'}$ is univalent). Since $\gamma'$ is an~unstable manifold for~$x$, there exists a~nested sequence of curves $\gamma_{i,*}'\supset\gamma_{i+1,*}'\ni{x}$, with $\gamma_{0,*}'=\gamma'$, such that $f^k$ maps $\gamma_{i+1,*}'$ diffeomorphically onto~$\gamma_{i,*}'$ and $\diam(\gamma_{i,*}')\to0$.

Now, consider a~linearization $\hat{\Psi}$ of~$f^k$ associated with~$x$, defined in a~neighbourhood of~$0$, i.e.
\begin{equation}	\label{xlin}
	f^k\circ\hat{\Psi}=\hat{\Psi}\circ\mu\,, \quad \textrm{where}\quad \mu=(f^k)'(x)\in\R \quad\textrm{and} \quad \hat{\Psi}(0)=x.
\end{equation}
Take $i$ so large that the~curve $\gamma_{i,*}'$ is short enough and there exists a~curve~$\hat{L}_i'\ni0$ which is mapped by~$\hat{\Psi}$ differomorphically onto~$\gamma_{i,*}'$. Note that
\begin{displaymath}
	f^{ik}|_{\gamma_{i,*}'}=\hat{\Psi}\circ\mu^i\circ\big(\hat{\Psi}|_{\hat{L}_i'}\big)^{-1}
\end{displaymath}
is a~diffeomorphism onto~$\gamma'$ and therefore $\hat{\Psi}$ restricted to the~curve $\hat{L}'=\mu^i\hat{L}_i'$ is also a~diffeomorphism and $\hat{\Psi}(\hat{L}')=\gamma'$. As $p\in\gamma'$, there exists $\hat{w}\in\hat{L}'$ such that $\hat{\Psi}(\hat{w})=p$.

Since $\gamma'$ is an~invariant manifold for~$x$ and $\hat{\Psi}$ is the~linearization associated with~$x$, the~curve~$\hat{L}'$ must be invariant under $z\to\mu{z}$. But the~only smooth curve through~$0$ invariant under multiplication by a~real number is a~line segment, thus $\hat{L}'$ must be contained in a~line~$\hat{M}$ through~$0$.

Let $z'=\T(z)\in\O_1$, where $z\in{L'}$ was the~point chosen from the~sequence $z_k'\to0$. For all large $j$ we have that $w_j=\lambda^{-jk}z'\in\O_0$. Note that $\Psi(w_j)\to{p}=\Psi(\hat{w})$ as $j\to\infty$. Since $\hat{\Psi}|_{\hat{L}'}$ is a~diffeomorphism and $\hat{w}\in\hat{L}'$, therefore for every $j$ large enough there exists near~$\hat{w}$ unique point $\hat{w}_j$ such that $\hat{\Psi}(\hat{w}_j)=\Psi(w_j)$.

Note that
\begin{displaymath}
	\hat{\Psi}(\mu^j\hat{w}_j)=f^{jk}\circ\hat{\Psi}(\hat{w}_j)=f^{jk}\circ\Psi(w_j)=\Psi(\lambda^{jk}w_j)=\Psi(z')=x.
\end{displaymath}
Let $\hat{M}_j'$ be the~line through $0$ and $\mu^j\hat{w}_j$ and let $\hat{M}_j\subset\hat{M}_j'$ be an~open line segment containing the~line segment $[\hat{w}_j,0]$ and contained in its small neighbourhood. Since $x$~is not an~iterate of a~singular value, by Lemma~\ref{regpreimages}, there exist neighbourhoods $\hat{\O}_0\ni0$ and $\hat{\O}_j\ni\mu^j\hat{w}_j$ on each of which $\hat{\Psi}$ is univalent and $\hat{\Psi}(\hat{\O}_0)=\hat{\Psi}(\hat{\O}_j)$. Next, applying Lemma~\ref{seqL} to the~map $\hat{\Psi}$ (taking the~line $\hat{M}_j'$ and the~point $\mu^j\hat{w}_j$ instead of $L$ and~$Q$), we obtain that $\hat{\Psi}(\hat{M}_j'\cap\hat{\O}_0)$ is an~invariant manifold for~$x$ and $\hat{\Psi}(\hat{M}_j'\cap\hat{\O}_0)=\hat{\Psi}(\hat{M}_j'\cap\hat{\O}_j)$.

Again by Lemma~\ref{seqL} and in view of~(\ref{xlin}), there exist small neighbourhoods $\hat{V}_j\ni\hat{w}_j$, $\hat{V}^1_j\ni\mu^j\hat{w}_j$ and $\hat{V}^0_j\ni0$ such that
\begin{equation}\label{inclus1}
	f^{jk}(\hat{\Psi}(\hat{M}_j'\cap\hat{V}_j))=\hat{\Psi}(\mu^j(\hat{M}_j'\cap\hat{V}_j))=\hat{\Psi}(\hat{M}_j'\cap\hat{V}^1_j)= \hat{\Psi}(\hat{M}_j'\cap\hat{V}^0_j) \subset \hat{\Psi}(\hat{M}_j).
\end{equation}
Since $\hat{w}_j$ lies close to~$\hat{w}$ and $\hat{\Psi}|_{[\hat{w},0]}$ is a~diffeomorphism, $\hat{\Psi}(\hat{M}_j)$ is a~smooth curve connecting the~points $x$ and~$\hat{\Psi}(\hat{w}_j)=\Psi(w_j)$ and lying close to~$\hat{\Psi}([\hat{w},0])=\Psi([0,z])$ (which is a~subarc of~$\gamma'$ connecting $p$ and~$x$). There exists therefore a~curve $M_j\subset\O_0$, passing through $w_j$ and~$z$, such that $\Psi(M_j)=\hat{\Psi}(\hat{M}_j)$. By~(\ref{inclus1}), there exists a~small neighbourhood~$V_j$ of~$w_j$ such that
\begin{equation}\label{inclus2}
	\Psi(\lambda^{jk}(M_j\cap{V_j}))=f^{jk}(\Psi(M_j\cap{V_j}))\subset \Psi(M_j)=\Psi(\T(M_j)).
\end{equation}

Denote by $\Ta_q(l)$ a~line tangent to a~curve~$l$ at a~point~$q$. Since $\lambda^{jk}w_j=z'$ and the~curves $\lambda^{jk}M_j$ and $\T(M_j)$ both pass through~$z'$, by~(\ref{inclus2}) these curves actually agree near~$z'$. In particular
\begin{equation}\label{tangent1}
	\Ta_{z'}(\lambda^{jk}M_j)=\Ta_{z'}(\T(M_j)).
\end{equation}
Note that $\Ta_{z'}(\lambda^{jk}M_j)=\Ta_{w_j}(M_j)$. Recall that $\hat{M}_j'$ is the~line containing $[\hat{w}_j,0]$, $\hat{w}_j\to\hat{w}$ as $j\to\infty$ and $\hat{M}$ is the~line through $0$ and~$\hat{w}$, therefore $\hat{M}_j'$ converges to~${\hat{M}}$. We also have $\hat{\Psi}(\hat{L}')=\gamma'=\Psi(L'\cap\O_0')$, where $\hat{L}'\subset\hat{M}$, and $\Psi(M_j)=\hat{\Psi}(\hat{M}_j)$ with $\hat{M}_j\subset\hat{M}_j'$. Since $\hat{\Psi}$~is a~diffeomorphism on a~neighbourhood of~$[0,\hat{w}]$, we get therefore that $M_j$~converges to a~line segment in~$L'$ in $C^1$-topology.

We conclude that $\Ta_{z'}(\lambda^{jk}M_j)=\Ta_{w_j}(M_j)\to\Ta_{0}(L')=\Ta_z(L')$ as $j\to\infty$ (since $z\in{L'}$) and also $\Ta_{z'}(\T(M_j))\to\Ta_{z'}(\T(L'))$ . This and~(\ref{tangent1}) imply that
\begin{displaymath}
	\Ta_z(L')=\Ta_{z'}(\T(L')),
\end{displaymath}
hence $\T'(z)\in\R$ as we claimed. Since it holds for a~sequence of points $z=z_k'\in{L'}$, $z_k'\to0$, we obtain that $\T'(z)\in\R$ for every~$z\in{L'}$.

Since there are infinitely many such lines passing through~$0$ on which $\T'$ is real, we conclude that $\T'$~must be constant. Thus $\T$~is an~affine map and the~identity $\Psi\circ\T=\Psi$ holds on the~whole~$\C$. But, as we mentioned in the~proof of Lemma~\ref{seqL}, it may happen only if it is of the~form $\T(z)=\pm{z}+Q$.
\end{proof}

Recall that, by Lemma~\ref{regpreimages}, $\Psi^{-1}(p)$~is a~regular fibre. As a~consequence of Lemma~\ref{Taffine}, for every pair $Q_1,Q_2\in\Psi^{-1}(p)$, there exists an~affine map $\T(z)=\pm{z}+c$ such that $\Psi\circ\T=\Psi$ on~$\C$ and $\T(Q_1)=Q_2$. It follows from \textit{(2')} of Definition~\ref{aut} that $\Psi$~is automorphic with respect to a~discrete group of isometries $\Gamma\subset\{\T:\T(z)=\pm{z}+c\}$.

This allows us to define a~measurable line field almost everywhere on~$\C$ by the~formula
\begin{displaymath}
	u(w)=\left(\frac{\psi'(z)}{|\psi'(z)|}\right)^{2} \quad \textrm{for any}\quad z\in\psi^{-1}(w).
\end{displaymath}
Note that, since $\Psi$ is automorphic with respect to~$\Gamma$, the~line field $u$ is well defined almost everywhere: if $z_1,z_2\in\psi^{-1}(w)$, then $\psi'(z_1)=\pm\psi'(z_2)$. Moreover, since $\lambda\in\R$, (\ref{lineariz})~immediately implies that 
\begin{displaymath}
	u(f(w))=\left(\frac{\psi'(\lambda{z})\lambda}{|\psi'(\lambda{z})\lambda|}\right)^{2}=\left(\frac{(f\circ\psi)'(z)}{|(f\circ\psi)'(z)|}\right)^{2}=\left(\frac{f'(w)}{|f'(w)|}\right)^2u(w).
\end{displaymath}
where $z\in\Psi^{-1}(w)$, hence $u$~is $f$-invariant. Thus, we have defined a~measurable $f$-invariant line field~$u$ on~$\C$ which is univalent in a~neighbourhood of the Julia set~$J(f)$ (cf.~{\cite[Definition~4.1]{rvs}}). However, it is impossible for a~transcendental function to posses such an~invariant line field what was proved in~\cite[Theorem~6.1]{rvs}.
This finishes the proof of Theorem~\ref{main}.

\end{document}